\documentclass[12pt]{article}
\usepackage{epsfig}
\usepackage[margin=1in]{geometry}
\usepackage{xcolor}
\usepackage{amsmath}
\usepackage{amsfonts}
\usepackage{amscd}
\usepackage{graphicx}
\usepackage{bbold}
\usepackage{pgfplots}
\usepackage{enumitem}
\usepackage{amsthm,amssymb}
\usepackage{mathtools}
\usepackage{thmtools}
\usepackage{blindtext}
\title{%Investigating
Effectiveness %Properties
of Walker's Cancellation Theorem}
\date{\today}
\author{Layth Al-Hellawi, Rachael Alvir, Barbara F. Csima, Xinyue Xie}

\DeclareMathOperator{\Q}{\mathbb{Q}}

\theoremstyle{plain}
\declaretheorem[name={Theorem},numberwithin=section]{theorem}
\declaretheorem[name={Lemma}, sibling=theorem]{lemma}

\declaretheorem[name={Claim}, numberwithin=theorem]{claim}

\theoremstyle{definition}

\declaretheorem[name={Question}]{question}

\begin{document}

\maketitle

\begin{abstract}
Walker's Cancellation theorem for abelian groups tells us that if $A$ is finitely generated and $G$ and $H$ are such that $A \oplus G \cong A \oplus H$, then $G \cong H$. Michael Deveau showed that the theorem can be effectivized, but not uniformly. In this paper, we expand on Deveau's initial analysis to show that the complexity of uniformly outputting an index of an isomorphism between $G$ and $H$, given indices for $A$, $G$, $H$, the isomorphism between $A \oplus G$ and $A \oplus H$, and the rank of $A$, is $\mathbf{0'}$. %Note: justify this statement in the preamble.
\end{abstract}

\section{Introduction}

Walker's Cancellation Theorem (WCT) was proved separately by Cohn \cite{Cohn} and Walker \cite{Walker} in 1956 using similar techniques, answering a question of Kaplansky's from his book \emph{Infinite Abelian Groups} \cite{Kaplansky}. The theorem concerns when it is possible to cancel the first term of isomorphic direct sums of two abelian groups. Here, the direct sum of two groups $G$ and $H$ will be written $G \oplus H$, and  throughout, all groups are abelian. The exact statement is as follows: 

%Rachael's Note: I changed the date you gave for Walker's theorem to 1956 since both the papers Deveau cites for them were published in 1956.

\begin{theorem}
Let $A$ be finitely generated, and let $G$ and $H$ be groups. If $A \oplus G \cong A \oplus H$, then $G \cong H$.
\end{theorem}

%Rachael's note: I deleted mention of the isomorphism f, since the statement $G \cong H$ does not mention it, i.e., the explicit name of the isomorphism is not actually used in the statement of the above theorem.

%Layth: Expanded out the section on motivating effectivization
Noah Schweber and Matthew Harrison-Trainor asked if Walker's result could be effectivized. That is, they asked if the theorem still holds under the assumption that $A,G,$ and $H$ are computable structures and the requirement that the resulting isomorphism is a computable function. For a discussion of common notions from computability theory, including computable structures and functions, see the book by Ash and Knight \cite{Knight}.

In his thesis \cite{Deveau}, Michael Deveau positively answered the non-uniform version of this question, and showed that a uniform result could be obtained under extra conditions. He also showed that these extra conditions were necessary, that is, that there was no uniform way to obtain the desired isomorphism effectively. Deveau then asked if the complexity of uniformity was some natural complexity class, such as that of $\mathbf{0'}$, which we investigate here. 

%Rachael's Note: I wanted to have more of a general discussion of what problem we are tackling and why in this paper before diving into the details. 

%\textcolor{orange}{(Introduction tells people what the paper is about, but also tries to sell it! What we did AND why it's interesting)}

%\textcolor{orange}{(try to over-cite, rather than under-cite, to make everyone happy! (e.g. Matthew Harrison Trainer, other literature on computable structure theory, other concrete examples of people studying stuff like this to fill up the references))}

%\textcolor{orange}{(reference a text for notation reasons; Soare for example, and a text for algebra notation)}

%\textcolor{orange}{(state the main theorem(s) in the introduction)}

%\textcolor{orange}{(on proofs: good practice to over-explain, then trim things later)}

%\textcolor{orange}{(put theorem in a theorem style)}

%Rachael's note: The below isn't substantially changed, but I think flows a little better in terms of sentence structure. The main improvement is mentioning that $E = A \oplus G $ AND $E = A \oplus H$ in the same breath. I also brought back in the language of identifying things with their images as Deveau used, which is common to use and I think makes clear what we're doing and why.
 
Before stating our results in full, we introduce a simplification of the problem. Instead of considering  $A \oplus G \cong A \oplus H$, we fix some copy $E$ of this isomorphism type and identify $A \oplus G$ and $A \oplus H$ with their images under some fixed isomorphisms to $E$. Note that $A$ may have different images under these isomorphisms, so we write $E = A \oplus G = B \oplus H$ with $A \cong B$ for clarity.

Deveau's first theorem can now be stated formally as follows:

\begin{theorem}
    Let $E$ be a computable group with  $E = A \oplus G = B \oplus H$, where $A, B$ are isomorphic and finitely generated, and that as sets, the subgroups $A$, $B$, $G$, and $H$ are computable relations on $E$. Then $G, H$ are computably isomorphic.  
%Suppose $A,B$ have cyclic decomposition $\bigoplus_{i \in n} \Z_{p_i^{k_i}} \oplus \Z^m$ where $a_i$, $b_i$ generate each component of $A, B$ respectively. 
    
%Then there is a computable functions which, given each $p_i, k_i, a_i, b_i$, and well as $n$ and $m$ as parameters, produces the index of a computable isomorphism between $G$ and $H$.

 %If each $p_i, k_i, a_i, b_i$, and well as $n$ and $m$ can be determined computably, then $G$ and $H$ are computably isomorphic. 
    
Suppose further that $A$ and $B$ are finite with generators $\bar{a},\bar{b}$ respectively. Then there
 is a single computable function $F$ which, when given  the indices for any such $E, A, B, G, H$, an index of a computable isomorphism $f: A \rightarrow B$, and the generators $\bar{a},\bar{b}$ as input, produces the index of a computable isomorphism between $G$ and $H$.
    
    %Original: Suppose further that $A$ and $B$ are finite with cyclic decomposition $\bigoplus_{i\in n} \Z_{p_i^{k_i}}$ where $a_i$, $b_i$ generate each component of $A, B$ respectively. Then there is a uniform computable function $F$ which, when given each $p_i, k_i, a_i, b_i$ and the indices for $A, B, G, H$, %and an index of a computable isomorphism $f: A \rightarrow B$ as input, produces the index of a computable isomorphism between $G$ and $H$.
\end{theorem}

%Rachael's note: Is this correctly stated? In particular I want to be oprecise about whether the inputs/outputs are, e.g., indices for functions or functions themselves - i.e., I found it unclear what it means to be ``given" information in certain cases- i.e., whether that information was to be given as a parameter or we only consider cases where that information is computably determined. 

%The big problem with the exact statement of the following is that it's not very formal. The first sentence isn't a complete sentence, and the whole result isn't written as an if-then statement.

Deveau's proof followed Cohn's proof while keeping track of what was computable along the way. His first result could not be made uniform because his construction of an isomorphism between $G$ and $H$ required hardcoded generators of cyclic subgroups of $G$ and $H$. However, in the case that $A$ and $B$ are finite, it is uniformly computable to find new decompositions $E=U\oplus G=U\oplus H$ where $U$ is some finite cyclic subgroup. Deveau also showed that, in the case that $A$ and $B$ are not necessarily finite, no such uniform computable $F$ exists, even allowing parameters that include the generators as above. This raises the question of what the complexity of such a uniform $F$ must be.  In this paper, we answer that question and show that the degree of uniformity is exactly $\mathbf{0}'$. In addition to the indices of the groups, the rank of $A, B$ is also provided as part of the input to $F$. Recall that the rank of an abelian group is the cardinality of a maximal linearly independent subset of the group \cite{Lang}. The main theorem that we will prove is formally stated as follows: 

% \begin{theorem}\label{main theorem}
% Let $E$ be a computable group with  $E = A \oplus G = B \oplus H$, where $A, B$ are isomorphic and finitely generated, and that as sets, the subgroups $A$, $B$, $G$, and $H$ are computable relations on $E$. Then the following are true:
% \begin{enumerate}
%     \item There is a $\mathbf{0}'$-computable function which, given the indices for $E, A, B, G, H$ and the rank of $A,B$, outputs an index for a computable isomorphism between $G$ and $H$. 
%     \item If $F$ is a uniform function which, given the above inputs for various such $E$, outputs an index for a computable isomorphism $f$ between $G$ and $H$, then $\mathbf{0}' \leq_T F$. %(Note: Good for A,B fg instead of cyclic.) 
% \end{enumerate}
% \end{theorem}
\begin{theorem}\label{main theorem}
We are interested in the setting where $E$ is a computable group, $E = A \, \oplus \, G = B \oplus H$, and $A$ and $B$ are isomorphic and finitely generated.

\begin{enumerate}
    \item There is a $\mathbf{0}'$-computable function $F(e,a,b,g,h,r)$ such that if $\mathcal{G}_e$ is a computable group $E$, and $C_a$, $C_b$, $C_g$, $C_h$ are computable sets that when inheriting the group operation from $E$ are themselves groups $A$, $B$, $G$, and $H$, satisfying $E = A \oplus G = B \oplus H$, and the groups $A$ and $B$ are isomorphic and finitely generated with rank $r$, then $\varphi_{F(e,a,b,g,h,r)}$ is an isomorphism from $G$ to $H$.
    \item Suppose that $F$ is a function such that whenever $\mathcal{G}_e$ is a computable abelian group $E$, and $C_a$, $C_b$, $C_g$, $C_h$ are computable sets that when inheriting the group operation from $E$ are themselves groups $A, B, G,$ and $H$, satisfying $E = A \oplus G = B \oplus H$, and the groups $A$ and $B$ are isomorphic and finitely generated with rank $r$, we have $\varphi_{F(e,a,b,g,h,r)}$ is an isomorphism from $G$ to $H$. Then $\emptyset' \leq_T F$. %(Note: Good for A,B fg instead of cyclic.) 
\end{enumerate}
\end{theorem}

The proof consists of two parts. In section 2, we prove (1). In section 3, we prove (2). %Maybe say some words about how we prive this? 

%The second part of the conclusion states that $F$ is Turing equivalent to the halting set $\emptyset'$. 

%The proof will be broken down into two parts. In Section 3, we prove that $F\le_T\emptyset'$ by describing a $\emptyset'$-computable oracle machine program that satisfies the requirements of $F$. The first part of the theorem follows naturally from this proof. In Section 4, we prove that $\emptyset'\le_T F$ by showing how $F$, together with some carefully designed inputs, can be used to compute the halting set. 

%Rachael's note: Are we using \emptyset' or \mathbf{0}'? We need to pick one and use it consistently. 

\section{ Upper Bound}

%Since WCT provides the aforementioned algorithm $F$, we can examine this algorithm closely to denote when oracles are necessary to complete a certain step, and if so, what kind of oracle. In this manner, we achieve an upper bound of complexity of $F$, according to its Turing degree. In this manner, we discover that $F \leq_T \emptyset'$

%The first thing that the algorithm does is assign $D := G \cap H$ to be the intersection of the two groups. The second step is to determine whether the quotient groups $G/D$ and $H/D$ are finite or infinite. Normally, such a question would be $\emptyset''$, but we can exploit the abelian group structure to make this question $\emptyset'$. The goal is to split up into a finite and infinite case on the quotient groups $G/D$ and $H/D$ (respectively, $U$ and $V$). 

% (Rachael will also rewrite the below:)

%In this section, we seek an upper bound on the Turing degree of the function $F$.

%Rachael's Note: We don't want to mention anything involving notation like $F$ without explicitly defining it, and we defined it in a previous section. In new sections, or in new theorem environments, etc, we should redefine any notation.

%We prove that the function $F$ is Turing reducible to the halting set $\emptyset'$. This is a stronger version of Deveau's Theorem 4.1, which follows the proof of WCT by Cohn \cite{Cohn}.

%Rachael's Note: The above is all stuff we've previously stated in the introduction, so it is redundant. 

In this section, we prove (1) of Theorem \ref{main theorem}, which roughly speaking gives an upper bound on the complexity of uniformity.

 We first prove the following theorem, which concerns a computable group $E = A \oplus G = B \oplus H$ where $A, B$ are cyclic. Note that here, the rank of $A$ and $B$ are not given as parameters. It is a stronger result than what is needed to prove (1) of Theorem \ref{main theorem}.

%Original: We first prove the following theorem, which focuses on the special case where $A$ and $B$ are known to be cyclic.

%Rachael's Note: Again, you're referring to A,B as if they are already known and defined, which happened in a previous section. You can stilll essentially refer to A,B but just word it slightly differently to make sure when you do your statement is "self-contained," which I've done above. 

 %We will later use this to prove (1) of Theorem \ref{main theorem}.

%Original: Note that is this restricted case, the hypothesis is relaxed by removing the requirement of the generators of $A$ and $B$ as part of the inputs of $F$. We will later use this to prove the more general result stated in (1) of Theorem \ref{main theorem}.

%Rachael's Note: "Restricted" and "Relaxed" are kind of opposite terms, so you definitely don't want to use them here interchangeably let alone right next to each other. We can actually avoid this language altogether and just say in a straightforward way what's been changed about the theorem. Also I've removed explicit reference to F for the same reason as above. I've also removed the line that says we will use this result to prove (1) of Thm 1.3. This is because we literally just said this a few sentences ago, so it's redundant to state again.

\begin{theorem}\label{cyclic theorem}
There is a $\mathbf{0}'$-computable function $F(e,a,b,g,h,r)$ such that if $\mathcal{G}_e$ is a computable group $E$, and $C_a$, $C_b$, $C_g$, $C_h$ are computable sets that when inheriting the group operation from $E$ are themselves groups $A$, $B$, $G$, and $H$, satisfying $E = A \oplus G = B \oplus H$,  and the groups $A$ and $B$ are isomorphic and cyclic, then $\varphi_{F(e,a,b,g,h,r)}$ is an isomorphism from $G$ to $H$.
\end{theorem}

% \begin{proof}

%Rachael: I really don't think we should cite the proof of a lemma we haven't proven yet. Still contemplating how to reword that. 

%Perhaps this should not inside a proof environment, but rather a preamble to a series of lemmas we need to prove before we give the proof of the theorem proper. In general, lemmas should be proved outside of the proofs of other theorems. Statements proved inside a proof of a theorem should be labeled as "claims." 

The theorem will be proven in three steps. In Lemma \ref{decomposing G and H} below, we start by using group-theoretic techniques similar to those used by Deveau and Cohn to decompose the groups $G$ and $H$. Further computability-theoretic arguments are made in the proof of Lemma \ref{find u and v} to show how a $\mathbf{0'}$-oracle suffices to provide all necessary information. Finally,
 in Lemma \ref{existence of computable isomorphism}, we define the desired computable isomorphism between $G$ and $H$.

We now introduce a useful decomposition of the groups $G$ and $H$. Let $D$ be the intersection of the two groups. In particular, we show that each of $G$ and $H$ is a direct sum of $D$ and some cyclic subgroup. This decomposition will later be used to define a computable isomorphism, which maps the generators of the
 two cyclic groups to each other and is the identity on $D$.

\begin{lemma}\label{decomposing G and H}
Suppose $E$ is a group with $E = A \oplus G = B \oplus H$, where $A$ and $B$ are isomorphic and cyclic. Then there exist $u\in G$ and $v\in H$ such that $G=U\oplus D$\footnote{ Strictly speaking, by this we mean that $G$ is equal to the internal direct product of $U$ and $D$; similarly for $H$.}$\cong V\oplus D=H$, where $D=G\cap H$, $U=\langle u\rangle$ and $V=\langle v\rangle$.
\end{lemma}

\begin{proof}
Define $D=G\cap H$. By the second isomorphism theorem,
\[G/D\cong G/(G\cap H)\cong (G+H)/H\le E/H\cong B\]

Since $B$ is cyclic, $G/D$ must also be cyclic. Therefore, there exists an element $u\in G$ such that $G/D=\langle u+D\rangle$. Define $U=\langle u\rangle$ so that $G=U\oplus D$. Similarly, $H/D$ is cyclic. Define $v\in H$ and $V=\langle v\rangle$ analogously so that $H=V\oplus D$. Then, $E=A\oplus U\oplus D=B\oplus V\oplus D$. Taking a quotient by $D$ yields $A\oplus U\cong B\oplus V$. If $A$ and $B$ are infinite cyclic groups, then $U$ and $V$ are either both infinite cyclic or both trivial, and hence are isomorphic. If $A$ and $B$ are finite cyclic groups, then $U$ and $V$ are both finite. We have $|A||U|=|A\oplus U|=|B\oplus V|=|B||V|$. Since $A\cong B$, we have $|A|=|B|$, which implies $|U|=|V|$. Hence, $U\cong V$. Therefore, $G=U\oplus D\cong V\oplus D=H$.
\end{proof}

Next, we prove that there is a $\mathbf{0'}$-computable procedure to find the generators $u$ and $v$ of the groups $U$ and $V$ as defined above. 
 The proof outlines the main steps of the procedure and shows that each step is $\mathbf{0'}$-computable. We proceed in two cases, depending on whether $U$ and $V$ are finite.

Throughout, we use the notation $n \cdot u$ as an abbreviation for the element $u$ added to itself $n$ times, to distinguish it from the multiplication of $n$ and $u$ as natural numbers. Note that since $E$ is computable, it is always computable to find an element $x$ such that $x = n \cdot u$.

\begin{lemma}\label{find u and v}
There is a $\mathbf{0'}$-computable function which, given indices for $E$, $A$, $B$, $G$ and $H$, outputs $u$ and $v$ as defined in Lemma \ref{decomposing G and H}.
\end{lemma}

\begin{proof}
%Recall that a $\mathbf{0'}$-oracle can answer any $\Sigma^0_1$ question \cite{Soare}. (Rachael: This is not actually in Soare, not stated like this. But it is well-known.)
We describe a $\mathbf{0'}$-computable procedure which, with access to the given parameters, can output $u$ and $v$. % Recall that since the halting set is Turing complete, every c.e. set is $\mathbf{0'}$-computable \cite{Soare}. It therefore suffices to show that each of the following steps in the procedure can be completed just by querying membership of c.e. sets. 

%Rachael: The above actually ends up being not used in the proof. 

The first step is to decide if $G/D$, as defined in Lemma 2.2, is finite. Recall that a nontrivial cyclic group is finite if and only if a non-identity element has finite order.  Using this, we obtain the following.  

%We need to be a bit careful with our language here, since $U$ is not even defined until we have found $u$. What we are finding is NOT a generator of u. We are finding a $u$ such that $u + D$ generators G/D, so that we know how to define $U$ in the first place. I have REMOVED mention of U and V, since they are NOT defined until we have found u,v and replaced it with mention of G/D. 

%Rachael: I have to reiterate that these are not, strictly speaking, $\Sigma_1$ sentences in the language of arithmetic, so they are not simply "Sigma_1 questions" that we can answer. So we have to do a little magic. 

\textbf{Step 1:} Decide if $G/D$ is finite and nontrivial. 

We note that this set is finite iff the following fact is true of $E$. 

\begin{align*}
G/D\text{ is finite but not trivial }\iff&\exists x\in G \, \exists n \in \omega [n \cdot (x+D)=D\land x \not \in D \land n \neq 0] \\
\iff&\exists x\in G \, \bigvee_{n > 0}[n \cdot x\in D\land x \not \in D ]
\end{align*}
%it's touch clearer to just write "u \not \in D" than "u \in G-H". Also we want u to be some particular element, so we better use a different letter here like "x"

To decide if this fact is true of $E$, first consider the following computable process. Begin by listing all the triples $(x, d, n)$ where $x$ in $G-D$, $d \in D$, and $n > 0$, which we can do since the sets $G$ and $H$ are both computable. As we list, we continually check the atomic diagram of $E$ to see if $n \cdot x = d$ is true of $E$. If so, we then halt - and if not this process continues forever. However, the halting set as oracle can tell us whether this process halts.\footnote{Throughout this paper, similar techniques suffice to show that $\mathbf{0'}$ can decide such statements that are not technically $\Sigma_1$ sentences in the language of arithmetic nor $\Sigma_1$ sentences in the language of groups.}

%PLEASE NOTE: G and D aren't definable by quantifier-free formulas (UGH) since they are not necessarily E-computable in EVERY copy of E, so this is ALSO not a $\Sigma_1$ statement in the language of groups. I think we can definitely get away with the idea that "everything is essentiall c.e." but we should spell it out formally why at least the first time. 

We now proceed to find the elements $u$ and $v$. We will describe how to find $u$; the process for finding $v$ is analogous. We show below that $\emptyset'$ as oracle can decide for a given $u$ if $u + D$ is not a generator of $G/D$. Depending on the outcome of step 1, we do this in one of two ways. Since $G/D$ actually is cyclic, to find a desired $u$ we merely test elements until the condition is met.

%We need to be a bit careful with our language here, since $U$ is not even defined until we have found $u$. What we are finding is NOT a generator of U. We are finding a $u$ such that $u + D$ generators G/D, so that we know how to define $U$ in the first place.

\textbf{Step 2:} Find a generator $u + D$ of $G/D$.

Case 1: $G/D$ is finite but not trivial. Then, for all $u\in (G-H)\cup\{0\}$,
\begin{align*}
&u+D\text{ is not a generator of }G/D\\
\iff&\exists\tilde{u}\in G-H \, [\tilde{u}+D\notin\langle u+D\rangle]\\
\iff&\exists\tilde{u}\in G-H \, \exists n \in \omega [n \cdot (u+D)=D\land n\neq 0\land\forall0\le m<n[\tilde{u}+D\neq m \cdot (u+D)]]\\
\iff&\exists\tilde{u}\in G-H \, \exists n \in \omega [n \cdot u\in D\land n\neq 0\land\forall0\le m<n \, [(\tilde{u}-m \cdot u)\notin D]]
\end{align*}
As in step 1, the halting set can determine if this fact is true of $E$. \\

Case 2: $G/D$ is infinite or trivial.

Recall that $G/D$ is cyclic, and that in an infinite cyclic group, only the generators are not divisible by some natural number $n$ for $n > 1$. Therefore, for all $u\in (G-H)\cup\{0\}$,
\begin{align*}
&u+D\text{ is not a generator of }G/D\\
\iff&\exists\tilde{u}\in G-H \, \exists n \in \omega[u+D=n \cdot (\tilde{u}+D)\land n > 1]\\
\iff&\exists\tilde{u}\in G-H \, \exists n \in \omega [(u-n \cdot \tilde{u})\in D\land n > 1]
\end{align*}

%Rachael minor note: -1 is not a natural number, and so the original way this was written above is a bit confused. For example, -3 \cdot u never makes any sense, since you can't take u added to itself -3 times. However, there is an element which is the inverse of the element 3 \cdot u. So I have changed things to account for this. 

%\textbf{Step 3:} Find a generator $v$ of $V$.
%The $\Sigma^0_1$ statements are analogous to the ones in step 2, with $G$ and $H$ interchanged.
%Rachael: We already said above that finding v is analogous and it suffices to describe how to find u, so explaining this again is redundant.
\end{proof}

Once $u$ and $v$ are found, we can build a computable isomorphism between $G$ and $H$, as described in the following proof.

\begin{lemma}\label{existence of computable isomorphism}
Suppose $E$ is a computable abelian group with $E = A \oplus G = B \oplus H$, where $A$ and $B$ are isomorphic and cyclic. Then uniformly in indices for $E$, $A$, $B$, $G$, $H$ and some $u$ and $v$ that are as guaranteed by Lemma \ref{decomposing G and H}, we can give a computable isomorphism between $G$ and $H$.
\end{lemma}

\begin{proof}
Define the following function $f:G\to H$. Given any $g\in G$, since $G=U\oplus D$ by Lemma \ref{decomposing G and H}, there exists a unique $k\in\omega$ such that $g=ku+d$ for some $d\in D$. We define $f(g)=kv+d$, which is in $H$ because $H=V\oplus D$ by Lemma \ref{decomposing G and H}. One can easily check that $f$ is an isomorphism. Moreover, since $u$ and $v$ are known, $D$ is computable, and $k$ is guaranteed to exist, $f$ can search for $k$ until $(g-ku)\in D$. Therefore, $f$ is a computable isomorphism.
\end{proof}

Now, we are ready to use the three lemmas above to prove Theorem \ref{cyclic theorem}.

\begin{proof}[Proof of Theorem \ref{cyclic theorem}]
We have lemma \ref{decomposing G and H} that guarantees a nice decomposition of $G=\langle u\rangle\oplus D$ and $H=\langle v\rangle\oplus D$. Lemma \ref{find u and v} presents a method for finding $u$ and $v$ using a $\mathbf{0'}$-oracle.
 Lastly, knowing $u$ and $v$, a computable isomorphism can be constructed, as in the proof of Lemma \ref{existence of computable isomorphism}. Therefore, $F$ is $\mathbf{0'}$-computable.
\end{proof}

%Rachael bookmark: The above is all correct. 

In order to generalize from the case where $A$ and $B$ are cyclic to the case where $A$ and $B$ are finitely generated, we need to use the fundamental theorem of finitely generated abelian groups, which guarantees a cyclic decomposition of $A$ and $B$. The following lemma shows that if the rank of $A$ and $B$ is known, then it is $\mathbf{0'}$-computable
 to find such a decomposition by computing a corresponding generating set.

\begin{lemma}\label{cyclic decomposition}
If $A$ is a computable finitely generated abelian group with known rank $r$, then it is $\mathbf{0'}$-computable from an index for $A$ and the rank, to find generators $a^f_1, a^f_2, \ldots, a^f_k, a_1, a_2, \ldots, a_r$ of the invariant factor decomposition of $A$, that is, $A\cong\left(\bigoplus_{1\le i\le k}\langle a^f_i\rangle\right)\oplus\left(\bigoplus_{1\le i\le r}\langle a_i\rangle\right)$, for some $k$, where $q_i\coloneqq|a^f_i|<\infty$ for $i=1, 2, \ldots, k$, $|a_i|=\infty$ for $i=1, 2, \ldots, r$, and $q_i|q_{i+1}$ for $i=1, 2, \ldots, k-1$.
\end{lemma}

%Rachael: generators for this decomposition are not unique, so I have erased the word "the" in front of the word "generators." 

%Rachael: Please remember to use the latex command "\ldots" instead of "..."

%Rachael note: I highly suggest changing the $f$'s to primes: instead of $a_1^f$ I would write $a_1'.$ 

\begin{proof}
Since $A$ is finitely generated, by the fundamental theorem of finitely generated abelian groups, $A$ has a unique decomposition $A\cong\mathbb{Z}_{q_1}\oplus\mathbb{Z}_{q_2}\oplus...\oplus\mathbb{Z}_{q_k}\oplus\mathbb{Z}^r$ where $q_i|q_{i+1}$ for $i=1, 2, ... k-1$ and $r$ is the rank. The following $\mathbf{0'}$-computable algorithm produces a set of generators for such a decomposition.

We first build the torsion subgroup $A_f$ of $A$. We start by defining $A_{f,0}$ to be the empty set. At stage $s > 0$, suppose we have defined a finite set $A_{f,s-1}$ of torsion elements of $A$. If $\exists a\in A - A_{f,s-1}\,\exists n \in \omega [n \cdot a= id]$ is true, we enumerate $A - A_{f,s-1}$ until a torsion element $a$ is found and let $A_{f,s}=A_{f,s-1}\cup\{a\}$. Otherwise, we stop, and by construction, we must have $A_f=A_{f,s-1}$. This algorithm terminates after finitely many steps since $A_f\cong\mathbb{Z}_{q_1}\oplus\mathbb{Z}_{q_2}\oplus...\oplus\mathbb{Z}_{q_k}$ must be finite.

It is computable to find an invariant factor decomposition of $A_f$. One obvious method is to enumerate all subsets of $A_f$ and find a minimal generating set that satisfies the invariant factor requirements. Let $A_f\cong\bigoplus_{1\le i\le k}\langle a^f_i\rangle$ be the invariant factor decomposition of $A_f$.

Let us now show that it is $\mathbf{0'}$-computable to find the generators for the infinite components of $A$. We would like to construct a $\Sigma^0_1$ statement, and the idea is similar to step 2, case 2, in the proof of Lemma \ref{find u and v}. Instead of testing a single element, we test all subsets $\{a_1, ..., a_r\}$ of $A-A_f$ with size $r$. 

%Rachael bookmark: above is correct.

\begin{claim}
For $a_1, a_2, ..., a_r\in A - A_f$, $\{a^f_1, ...a^f_k, a_1, ..., a_r\}$ is not a generating set of $A$ if and only if there exist elements $\tilde{a}_1, ..., \tilde{a}_r\in A-A_f$ and matrices $M\in\mathcal{M}_{r\times k}(\mathbb{Z})$ and $N\in\mathcal{M}_{r\times r}(\mathbb{Z})$ such that $\vec{a}=M\vec{a^f}+N\vec{\tilde{a}}$ and $det(N)\neq\pm1$, where $\vec{a^f}=(a^f_1,... , a^f_k)$, $\vec{a}=(a_1, ..., a_r)$, and $\vec{\tilde{a}}=(\tilde{a}_1, ..., \tilde{a}_r)$.
\end{claim}

\begin{proof}
We first prove the forward direction. Suppose $\{a^f_1, ...a^f_k, a_1, ..., a_r\}$ is not a generating set of $A$. Let $\{a^f_1, ..., a^f_k, \tilde{a}_1, ..., \tilde{a}_r\}$ be a generating set of $A$. Define $\vec{a^f}=(a^f_1,... , a^f_k)$, $\vec{a}=(a_1, ..., a_r)$, and $\vec{\tilde{a}}=(\tilde{a}_1, ..., \tilde{a}_r)$. Then, there are two integer matrices $M\in\mathcal{M}_{r\times k}(\mathbb{Z})$ and $N\in\mathcal{M}_{r\times r}(\mathbb{Z})$ such that
\begin{equation}\label{matrix eq}
\left[
\begin{array}{c}
\vec{a^f}\\ \hline\vec{a}
\end{array}
\right]=\left[
\begin{array}{c|c}
\mathbb{1}_{k\times k} & \mathbb{0}_{k\times r}\\ \hline M & N
\end{array}
\right]\left[
\begin{array}{c}\vec{a^f}\\ \hline \vec{\tilde{a}}
\end{array}
\right]
\end{equation}

Using the fact that an integer matrix has an integer inverse if and only if its determinant is $\pm1$, we have that
\begin{align*}
&\{a^f_1, ..., a^f_k, a_1,... , a_r\}\text{ is not a generating set of }A\\
\iff&\left[\begin{array}{c|c}
\mathbb{1}_{k\times k} & \mathbb{0}_{k\times r}\\
\hline
M & N
\end{array}\right]\text{ is not invertible in } \mathcal{M}_{(k+r)\times (k+r)}(\mathbb{Z})\\
\iff&N\text{ is not invertible in }\mathcal{M}_{r\times r}(\mathbb{Z})\\
\iff&\det(N)\neq\pm1
\end{align*}

For the converse direction, we prove its contrapositive. Suppose $\{a^f_1, ...a^f_k, a_1, ..., a_r\}$ is a generating set of $A$. For any elements $\tilde{a}_1, ..., \tilde{a}_r\in A-a_f$ and matrices $M\in\mathcal{M}_{r\times k}(\mathbb{Z})$ and $N\in\mathcal{M}_{r\times r}(\mathbb{Z})$ such that $\vec{a}=M\vec{a^f}+N\vec{\tilde{a}}$, where $\vec{a^f}=(a^f_1,... , a^f_k)$, $\vec{a}=(a_1, ..., a_r)$, and $\vec{\tilde{a}}=(\tilde{a}_1, ..., \tilde{a}_r)$, we want to show that $\det(N)=\pm1$. Note that there exist matrices $M'\in\mathcal{M}_{r\times k}(\mathbb{Z})$ and $N'\in\mathcal{M}_{r\times r}(\mathbb{Z})$ such that
\begin{align*}
&\left[
\begin{array}{c}
\vec{a^f}\\\hline\vec{\tilde{a}}
\end{array}
\right]=\left[
\begin{array}{c|c}
\mathbb{1}_{k\times k} & \mathbb{0}_{k\times r}\\ \hline M' & N'
\end{array}
\right]\left[
\begin{array}{c}\vec{a^f}\\ \hline \vec{a}
\end{array}
\right]\\
\implies &\left[
\begin{array}{c}
\vec{a^f}\\ \hline\vec{a}
\end{array}
\right]=\left[
\begin{array}{c|c}
\mathbb{1}_{k\times k} & \mathbb{0}_{k\times r}\\ \hline M & N
\end{array}
\right]\left[
\begin{array}{c|c}
\mathbb{1}_{k\times k} & \mathbb{0}_{k\times r}\\ \hline M' & N'
\end{array}
\right]\left[
\begin{array}{c}\vec{a^f}\\ \hline \vec{a}
\end{array}
\right]=\left[
\begin{array}{c|c}
\mathbb{1}_{k\times k} & \mathbb{0}_{k\times r}\\ \hline M+NM' & NN'
\end{array}
\right]\left[
\begin{array}{c}\vec{a^f}\\ \hline \vec{a}
\end{array}
\right]
\end{align*}

Since $\{a^f_1, ..., a^f_k, a_1, ..., a_r\}$ is a minimal generating set of $A$, we have $NN'=\mathbb{1}$. So $N$ is invertible in $\mathcal{M}_{r\times r}(\mathbb{Z})$ and hence $\det(N)=\pm1$.
\end{proof}

Using the claim, we have the following $\Sigma^0_1$ statement that tests if a subset of $A-A_f$ with size $r$ is the desired set of generators.
\begin{align*}
&\text{For }a_1, a_2, ..., a_r\in A - A_f, \{a^f_1, ...a^f_k, a_1, ..., a_r\}\text{ is not a generating set of }A\\
\iff &\forall 1\le i\le r \, \exists\tilde{a}^f_i\in A_f \, \forall 1\le j\le r \, \exists\tilde{a}_j\in A - A_f \, \exists n_{ij}\\
&\left[\forall 1\le l\le r\left[a_l=\tilde{a}^f_l+\sum_{1\le j\le r} n_{lj} \cdot \tilde{a}_j\right]\land \, \det((n_{st})_{1\le s\le r, 1\le t\le r})\neq\pm 1\right]
\end{align*}

\end{proof}

Finally, we combine Theorem \ref{cyclic theorem} and Lemma \ref{cyclic decomposition} to prove part 1 of Theorem \ref{main theorem}

\begin{theorem}\label{main theorem part 1}
There is a $\mathbf{0}'$-computable function $F(e,a,b,g,h,r)$ such that if $\mathcal{G}_e$ is a computable abelian group $E$, and $C_a$, $C_b$, $C_g$, $C_h$ are computable sets that when inheriting the group operation from $E$ are themselves groups $A$, $B$, $G$, and $H$, satisfying $E = A \oplus G = B \oplus H$, and the groups $A$ and $B$ are isomorphic and finitely generated with rank $r$, then $\varphi_{F(e,a,b,g,h,r,i)}$ is an isomorphism from $G$ to $H$.
\end{theorem}

\begin{proof}
Since $A\cong B$, by uniqueness (up to isomorphism) of invariant factor decomposition, we have that $A\cong\mathbb{Z}_{q_1}\oplus\mathbb{Z}_{q_2}\oplus...\oplus\mathbb{Z}_{q_k}\oplus\mathbb{Z}^r\cong B$ where $q_i|q_{i+1}$ for $i=1, 2, ... k-1$ and $r$ is the rank, which is given as part of the inputs of $F$. By Lemma \ref{cyclic decomposition}, it is $\mathbf{0'}$-computable to find the generators $a^f_1, a^f_2, ..., a^f_k, a_1, a_2, ..., a_r$ of group $A$ as defined in the lemma, as well as the corresponding generators $b^f_1, b^f_2, ..., b^f_k, b_1, b_2, ..., b_r$ of group $B$. Let us define $A_i=\langle a^f_i\rangle$, for $1\le i\le k$ and $A_i=\langle a_i\rangle$ for $k+1\le i\le k+r$. Define $B_i$ for $1\le i\le k+r$ analogously. Then, for $1\le i\le k+r$, $A_i\cong B_i$ and both are cyclic. We have
\begin{equation}\label{original E}
E=A_1\oplus A_2\oplus...\oplus A_{k+r}\oplus G=B_1\oplus B_2\oplus...\oplus B_{k+r}\oplus H
\end{equation}

Define $G'=A_2\oplus...\oplus A_{k+r}\oplus G$ and $H'=B_2\oplus...\oplus B_{k+r}\oplus H$. Note that $G'$ and $H'$ are computable subgroups of $E$. Since $A_1, B_1$ are isomorphic and cyclic, applying Theorem \ref{cyclic theorem} to $E, A_1, B_1, G', H'$, we can find a computable isomorphism $f_1$ between $G'$ and $H'$ using a $\mathbf{0'}$ oracle. Then,
\[H'=f_1(A_2)\oplus...\oplus f_1(A_{k+r})\oplus f_1(G)=B_2\oplus...\oplus B_{k+r}\oplus H\]

Notice that this is of the same form as equation (\ref{original E}) with one less component, which allows us to apply Theorem \ref{cyclic theorem} again. After repeating this process $(k+r)$ times, with $f_i$ being the isomorphism obtained at the $i$th step, we eventually arrive at a computable isomorphism
 $(f_{k+r}\circ f_{k+r-1}...\circ f_1)$ between $G$ and $H$. 
\end{proof}

%Rachael bookmark

\section{Lower Bound}
In this section, we prove (2) of Theorem \ref{main theorem}, which is restated below. It gives, roughly speaking, a lower bound on the complexity of uniformity. In addition, we will prove a stronger result which requires the generators of $A$ and $B$ to be part of the input of $F$, in addition to their rank.

\begin{theorem}\label{part(2) with rank}
Suppose $F$ is a function such that whenever $\mathcal{G}_e$ is a computable abelian group $E$, and $C_a$, $C_b$, $C_g$, $C_h$ are computable sets that when inheriting the group operation from
 $E$ are themselves groups $A, B, G,$ and $H$, satisfying $E = A \oplus G = B \oplus H$, and the groups $A$ and $B$ are isomorphic and finitely generated with rank $r$, we have that $\varphi_{F(e,a,b,g,h,r)}$ is an isomorphism from $G$ to $H$. Then $\emptyset' \leq_T F$. %(Note: Good for A,B fg instead of cyclic.) 
\end{theorem}

\begin{proof}
We prove the theorem by assuming such an $F$ exists and showing it can compute the halting set $K$. Computably, we will construct for every $e\in\omega$ specific abelian groups $E_e=A_e\oplus G_e = B_e\oplus H_e$, where $A$ and $B$ are isomorphic and finitely generated.
%Rachael note: Technically $e$ has already been used as the index for the computable group E. We need a different letter somewhere.  
We will construct $A_e$ and $B_e$ with rank 1 so that $F$ can use this information as input. In fact, $A_e$ and $B_e$ are infinite cyclic groups in the following proof.

We construct the atomic diagrams of the groups stagewise. The indices of the groups are then given by the Recursion Theorem \cite{Soare}. Let $K_s$ denote the set of the first $s$ elements enumerated. Two different instances are constructed depending on whether $e$ is in the halting set. We omit the subscripts $e$ for readability and just write, e.g., $A$ instead of $A_e$, keeping in mind that the construction of $A$ depends on $e$. 

At stage $s$ of the construction, we put atomic facts into the diagrams of $A_e, B_e, G_e, H_e$ and $E_e$ for $e \leq s$, and proceed as if we were enumerating the atomic diagrams of subgroups of $\Q^2$. When we begin enumerating the diagram of these groups for some $e$, we add in a trivial fact to ensure that the element $a = h = (1,0)$ is in both $A$ and $H$, and the element $b = g = (0,1)$ is in both $B$ and $G$. At every stage, we add in the necessary facts to the atomic diagram of $E$ to ensure that $E=A\oplus G=H\oplus B$, where $A = H$ and $G = B$.\footnote{Note that in this instance, we will ensure that $E$ is the internal direct product of $A$ and $G$, as well as the internal direct product of $B$ and $H$. Note that since $E$ is abelian and we are dealing with internal direct products, $B \oplus H = H \oplus B$.} 

At stage $s$, if $e \not \in K_s$, we continue to build $A,H$ as if $A = H = \langle (1,0) \rangle \cong \mathbb{Z}$ and $B,G$ as if $B = G = \langle (0,1) \rangle \cong \mathbb{Z}$. More precisely, we introduce elements $(n,0)$ and $(0,n)$ for $n \leq s$ and add facts to the relevant atomic diagrams that ensure they are the intended integer multiples of $(1,0)$ and $(0,1)$. 
If $e \not \in K$, then in the limit, we have $E=A\oplus G=B\oplus H=\mathbb{Z}^2$, and any isomorphism $f:G\to H$ must satisfy $f((0,1))=\pm(1,0)$.

%For every $e\in\omega$, we define the following: 
%\[a=h=(1,0), b=g=(0,1), a\in A, g\in G, b\in B, h\in H\]

%As long as $e \notin K_s$, we only enumerate integer multiples of $a, b, g, h$ into their corresponding groups.

On the other hand, if $s$ is the first such that $e \in K_s$, then we introduce the elements $\hat{a}=\hat{h}=\left(\frac{1}{2}, 0\right)$ in ${A}$ and ${H}$, and add facts to the atomic diagrams to ensure they have the properties intended. We continue building $A,H$ as if $A = H = \left\langle \left(\frac{1}{2}, 0\right) \right\rangle \cong \mathbb{Z}$. We continue to build $B,G$ just as in the previous case. In the limit, for $e \in K$ we will have $E=A\oplus G=B\oplus H=\frac{1}{2}\mathbb{Z}\times\mathbb{Z}$, where any isomorphism $f:G\to H$ must satisfy $f((0,1))=\pm\left(\frac{1}{2},0\right)$.

For $e\in\omega$, let $f_e$ be the isomorphism provided by $F$ when applied to the groups constructed as above. Then,
\[f_e((0,1))=
\begin{cases}
\pm(1,0)\quad&\text{if }e \notin K\\
\pm\left(\frac{1}{2},0\right)\quad&\text{if }e\in K
\end{cases}\]

\end{proof}

% One basic result is that, if we allow ourselves uniformity in all generators of $A, B, G$, and $H$, then we get that this version of $F$ is above $\emptyset'$. To show this, We initially work with $a=h=(1,0)$ and $b=g=(0,1)$.
% Declare $a\in A$, $g\in G$, $b\in B$, $h\in H$  (noting that we allow $G$ and $H$ to also be cyclic in this case) and start enumerating elements that are integers multiples of the "generators" in each group. 

% This smaller proof demonstrates how we go about proving the case where we are only uniform in $G$ and $H$ (i.e. $a$ and $b$ as generators are fixed). We need to declare some elements to build $A, B, G, H$ in a case where we assume our index $e$ has not yet entered the halting set. Moreover, we have the stringent requirements that $A \cong B$, and that $A \oplus G = B \oplus H$ at the end. If $e$ does enter the halting set, we need to modify $A, B, G, H$ in such a way that, if you were given the isomorphism $f : G \to H$, you could pick at least one test element $x \in G$, and the result in the image $f(x)$, reveals whether $e \in K$. The additional challenge of the fixed $A, B$ and uniform $G, H$ is that, through $A \oplus G = B \oplus H$, any changes made to $H$ get reflected into $G$ (whereas, in the above proof, we could reflect them into $A$). The subsequent proof for the case of $a, b$ fixed will use an infinitely generated $G, H$ to get around this.

The construction in the above proof is relatively straightforward, as we simply changed the generators of groups $A$ and $H$, as needed, to code in the information of $K$. Let us now consider a more restrictive case where we require both the rank and the generators of $A$ and $B$ to be fixed in advance. A more complicated construction leads to the following stronger theorem. 

%Rachael: Fixed comma splice and passive language above.

\begin{theorem}\label{part(2) with generators}
Suppose that $F$ is a function such that whenever $\mathcal{G}_e$ is a computable abelian group $E$, and $C_a$, $C_b$, $C_g$, $C_h$ are computable sets that when inheriting the group operation from $E$ are themselves groups $A, B, G,$ and $H$, satisfying $E = A \oplus G = B \oplus H$, and the groups $A$ and $B$ are isomorphic and finitely generated with rank $r$, and with generators $a_1, \dots a_n, b_1, \dots, b_n$, we have that $\varphi_{F(e,a,b,g,h,r, a_1, \dots, a_n, b_1, \dots, b_n)}$ is an isomorphism from $G$ to $H$. Then  $\emptyset' \leq_T F$. %(Note: Good for A,B fg instead of cyclic.) 
\end{theorem}

\begin{proof}
We use similar notation as in the proof of Theorem \ref{part(2) with rank} and follow the same setup. That is, we consider an enumeration $K_s$ of the halting set and, for every $e\in\omega$, construct abelian groups $E=A\oplus G=B\oplus H$ stagewise in a way that depends on whether  $e \in K_s$. This time, however, we act as if we are enumerating subgroups of $\mathbb{Q}^3$.
%, such that
% \begin{itemize}
%     \item $E_e=A_e\oplus G_e=B_e\oplus H_e$,
%     \item $A_e\cong B_e$ are cyclic groups,
%     \item Given the indices of $E_e$, $A_e$, $B_e$, $G_e$, and $H_e$, it is computable to determine whether $e\in\emptyset'$ using the output of $F$.
% \end{itemize}
% Given a computable enumeration $K_s$ of the halting set, we will describe the group elements that are revealed concurrently with each step of the enumeration. The indices of the groups are then given by the Recursion Theorem \cite{Soare}. Two different instances will be construct depending on whether $e$ is in the halting set. We omit the subscripts $e$ for simplicity, and we use $\hat{}$ to denote all elements, groups, isomorphisms, etc. associated with the case $e\in\emptyset'$.

When we begin to enumerate the atomic diagrams of $A,B$, we ensure that the elements $a = (1,0,0)$ and $b = (2,1,2)$ are elements of $A,B$ respectively and continue to enumerate the atomic diagram as if $A=\langle a\rangle$ and $B=\langle b\rangle$ unless directed otherwise. When we begin to enumerate the diagrams of $G,H$ we ensure the elements $g_1=(0,1,0)$ and $g_2=(0,0,1)$ are in $G$, and $h_1=(5,2,5)$ and $h_2=(0,0,1)$ are in $H$. We continue to enumerate these atomic diagrams as if $G,H$ are generated only by explicitly introduced elements.

%We define and declare the following,
%\begin{align*}
%a&=(1,0,0)\quad A=\langle a\rangle\quad\quad\quad %g_1=(0,1,0)\quad g_2=(0,0,1)\quad g_1, g_2\in G\\
%b&=(2,1,2)\quad B=\langle b\rangle\quad\quad\quad %h_1=(5,2,5)\quad h_2=(0,0,1)\quad h_1, h_2\in H
%\end{align*}

%Let $A$ and $B$ be the infinite cyclic groups generated by $a$ and $b$, respectively, and enumerate their elements accordingly. We also assume that elements in $\langle\{g_1,g_2\}\rangle$ and $\langle\{h_1,h_2\}\rangle$ are properly enumerated into $G$ and $H$. 

If $e\notin K_s$ at stage $s$, we add in the following element to both $G$ and $H$:
\[N_s\coloneqq\left(0,0,\frac{1}{\prod_{i=1}^s n_i}\right)\]
Here, the sequence $(n_i)_{i\ge1}$ lists $\mathbb{N}-5\mathbb{N}$ in increasing order. 

At the end of any finite stage $s$, it would appear that $G$ and
 $H$ are finitely generated by $\{g_1,N_s\}$ and $\{h_1,N_s\}$, respectively. Note that $a=h_1-2b-h_2$ and $g_1=5b-2h_1$. So, if $e\notin K$, we have
\[E=A\oplus G=B\oplus H=\mathbb{Z}\times\mathbb{Z}\times\left\{\frac{p}{q}:p,q\in\mathbb{Z}, \gcd(p,q)=1\text{ and }5\nmid q\right\}.\]

If $e$ enters the halting set at some stage $s$, then we first find the largest $N$ such that $(0,0,1/N)$ has appeared in some previous stage, i.e. $N=\prod_{i=1}^{s-1}n_i$. Let $n=5N$. We then ensure the following elements appear in $G$ and $H$ respectively:
\begin{align*}
\hat{g}_1&=\left(0,\frac{1}{5},\frac{1}{5n}\right)\quad\quad\quad \hat{g}_2=\left(0,0,\frac{1}{n}\right)\quad\\
\hat{h}_1&=\left(1,\frac{2}{5},\frac{5n+2}{5n}\right)\quad \hat{h}_2=\left(0,0,\frac{1}{n}\right)\quad
\end{align*}

%We keep $\hat{A}$ and $\hat{B}$ the same as $A$ and $B$, as required by $F$. In all subsequent stages, we do not enumerate new elements into $\hat{G}$
 or $\hat{H}$ except for those generated by existing elements. 

Notice that
\begin{align*}
&\hat{g}_1=\frac{g_1+\hat{g}_2}{5}\quad\quad\quad\hat{h}_1=\frac{h_1+2\hat{h}_2}{5}\quad\quad\quad N_{s-1}=5\hat{g_2}=5\hat{h_2}\\
\implies&g_1,N_{s-1}\in\langle\{\hat{g}_1,\hat{g}_2\}\rangle\text{ and }h_1,N_{s-1}\in\langle\{\hat{h}_1,\hat{h}_2\}\rangle
\end{align*}

Therefore, when $e \in K$, in the limit we will have $G=\langle\hat{g}_1, \hat{g}_2\rangle$ and ${H}=\langle \hat{h}_1,\hat{h}_2\rangle$. We still have ${E}={A}\oplus{G}={B}\oplus{H}$ since
\[\begin{bmatrix}a\\\hat{g}_1\\\hat{g}_2\end{bmatrix}
=\begin{bmatrix}-2&5&-n-2\\1&-2&1\\0&0&1\end{bmatrix}
\begin{bmatrix}b\\\hat{h}_1\\\hat{h}_2\end{bmatrix}\quad\quad\text{ and }\quad\quad\begin{bmatrix}b\\\hat{h}_1\\\hat{h}_2\end{bmatrix}
=\begin{bmatrix}2&5&2n-1\\1&2&n\\0&0&1\end{bmatrix}
\begin{bmatrix}a\\\hat{g}_1\\\hat{g}_2\end{bmatrix}\]

We now show that $\emptyset' \leq_T F$. Using $F$, we will obtain two isomorphisms $f:G_{e'}\to H_{e'}$ and $\hat{f}:{G}_{e''} \to {H_{e''}}$, where $e' \not \in K$ and $e'' \in K$. 

Let us observe first how such $f$ and $\hat{f}$ act on the inputs $g_1$ and $g_2$. There exist $t_{11},t_{12},t_{21},t_{22},$ \newline $\hat{t}_{11},\hat{t}_{12},\hat{t}_{21},\hat{t}_{22}\in\mathbb{Z}$ such that
\begin{align}
&f(g_1)=t_{11}h_1+t_{12}\left(0,0,\frac{1}{q_1}\right)\quad&&\text{where }q_1\in\mathbb{Z}\text{ and }5\nmid q_1\nonumber\\
&f(g_2)=t_{21}h_1+t_{22}\left(0,0,\frac{1}{q_2}\right)=\left(0,0,\frac{t_{22}}{q_2}\right)\quad&&\text{where }q_2\in\mathbb{Z}\text{ and }5\nmid q_2\nonumber\\
&\begin{bmatrix}\hat{f}(\hat{g}_1)\\ \hat{f}(\hat{g}_2)\end{bmatrix}=
\begin{bmatrix}
\hat{t}_{11}&\hat{t}_{12}\\
\hat{t}_{21}&\hat{t}_{22}
\end{bmatrix}\begin{bmatrix}
\hat{h}_1\\\hat{h}_2
\end{bmatrix}\quad&&\text{where }\begin{vmatrix}
\hat{t}_{11}&\hat{t}_{12}\\
\hat{t}_{21}&\hat{t}_{22}
\end{vmatrix}=\hat{t}_{11}\hat{t}_{22}-\hat{t}_{12}\hat{t}_{21}=\pm1\label{det constraint}
\end{align}

For any $m\in\mathbb{Z}$ such that $5\nmid m$, we have $\frac{g_2}{m}\in G_{e'}$, so $\frac{f(g_2)}{m}\in H_{e'}$. Hence, $t_{21}=0$. The constraint on the determinant follows from the fact that $H_{e''}=\langle\{\hat{f}(\hat{g}_1),\hat{f}(\hat{g}_2)\}\rangle$, which implies that the coefficient matrix is invertible and its inverse only contains integer entries. Note that
\begin{align*}
\hat{f}(g_1)&=\hat{f}(5\hat{g}_1-\hat{g}_2)=5\hat{f}(\hat{g}_1)-\hat{f}(\hat{g}_2)=(5\hat{t}_{11}-\hat{t}_{21})\hat{h}_1+(5\hat{t}_{12}-\hat{t}_{22})\hat{h}_2\\
&=(5\hat{t}_{11}-\hat{t}_{21})\frac{h_1+2\hat{h}_2}{5}+(5\hat{t}_{12}-\hat{t}_{22})\hat{h}_2\\
&=\left(\hat{t}_{11}-\frac{\hat{t}_{21}}{5}\right)h_1+\left(2\hat{t}_{11}+5\hat{t}_{12}-\frac{2\hat{t}_{21}}{5}-\hat{t}_{22}\right)\hat{h}_2
\end{align*}

For any $e\in\omega$, we apply $F$ to the groups built as above. Let $f_e$ denote the isomorphism given by $F$. Then, we can decide if $e\in K$ as follows.

If $f_e(g_2)=(x,y,z)$ where $x,y\neq0$, then $f_e$ cannot be an isomorphism between $G_{e'}$ and $H_{e'}$, since $t_{21}=0$ is always true. So, $e\in K$. Otherwise, $f_e(g_2)=(0,0,z)$, then we compute $f_e(g_1)=(x',y',z')$. Note that
\begin{align*}
&f_e(g_2)=(0,0,z)\text{ and }f_e\text{ is an isomorphism between }G_{e''}\text{ and }H_{e''}\\
\implies &\hat{t}_{21}=0\text{ and }f_e\text{ is an isomorphism between }G_{e''}\text{ and }H_{e''}\\
\implies &f_e(g_1)=\hat{t}_{11}h_1+(2\hat{t}_{11}+5\hat{t}_{12}-\hat{t}_{22})\hat{h}_{2}\text{ where }\hat{t}_{11},\hat{t}_{22}=\pm1\quad\text{(by (\ref{det constraint}))}\\
\implies&f_e(g_1)=\hat{t}_{11}h_1+\left(0,0,\frac{2\hat{t}_{11}+5\hat{t}_{12}-\hat{t}_{22}}{n}\right)\text{ where }5|n\text{ and }2\hat{t}_{11}+5\hat{t}_{12}-\hat{t}_{22}\not\equiv 0\pmod 5
\end{align*}

We write $(x',y',z')$ as $(x',y',z')=mh_1+\left(0,0,\frac{p}{q}\right)$ for some $m,p,q\in\mathbb{Z}$ and $\gcd(p,q)=1$. If $5|q$, then $(x',y',z')\notin H_{e'}$, so $f_e$ must be an isomorphism between $G_{e''}$ and $H_{e''}$. Hence, $e\in K$. Otherwise, $e\notin K$.

\end{proof}

Finally, combining the results proved in sections 2 and 3, we can prove our main theorem \ref{main theorem}, which is restated below.

\begin{theorem}
We are interested in the setting where $E$ is a computable abelian group, $E = A \, \oplus \, G = B \oplus H$, and $A$ and $B$ are isomorphic and finitely generated.

\begin{enumerate}
    \item There is a $\mathbf{0}'$-computable function $F(e,a,b,g,h,r)$ such that if $\mathcal{G}_e$ is a computable abelian group $E$, and $C_a$, $C_b$, $C_g$, $C_h$ are computable sets that when inheriting the group operation from $E$ are themselves groups $A$, $B$, $G$, and $H$, satisfying $E = A \oplus G = B \oplus H$,  and the groups $A$ and $B$ are isomorphic and finitely generated with rank $r$, then $\varphi_{F(e,a,b,g,h,r)}$ is an isomorphism from $G$ to $H$.
    \item Suppose that $F$ is a function such that whenever $\mathcal{G}_e$ is a computable abelian group $E$, and $C_a$, $C_b$, $C_g$, $C_h$ are computable sets that when inheriting the group operation from $E$ are themselves groups $A, B, G,$ and $H$, satisfying $E = A \oplus G = B \oplus H$, and the groups $A$ and $B$ are isomorphic and finitely generated with rank $r$, we  have that $\varphi_{F(e,a,b,g,h,r)}$ is an isomorphism from $G$ to $H$. Then $\emptyset' \leq_T F$. %(Note: Good for A,B fg instead of cyclic.) 
\end{enumerate}
\end{theorem}

\begin{proof}
Part 1 follows immediately from Theorem \ref{main theorem part 1}, and part 2 follows immediately from Theorem \ref{part(2) with rank}.
\end{proof}

\section{Open Problems}
%Changed ``Conclusion" title to ``Open Problems"

%Layth: Wrote up a conclusion; feel free to edit!

%We have shown that for any function $F$ that computes the isomorphism between $G$ and $H$ uniformly in $A, B, G, H$ as defined in Theorem \ref{main theorem}, with known rank of $A, B$, its lowest achievable Turing degree is exactly $\mathbf{0}'$. %$F =_T \mathbf{0'}$ in the case of $G$ and $H$ uniformity and $A, B, G, H$ uniformity in the infinite case.
%This generalizes Deveau's result, that knowing $A, B$ to be finite, i.e., $A, B$ have rank 0, there exists an $F$ with Turing degree $\mathbf{0}$. Theorem \ref{cyclic theorem} also points to the fact that for cyclic groups $A, B$, even if the rank is unknown, we still have $F \leq_T \mathbf{0'}$. These results left some open problems. 

%R: Commented out the summary of our results. We stated them very clearly in the introduction and also proved them, so stating them again a third time seems redundant. We can simply refer back to them at this point. 

There remain several questions about whether or not the theorems proven in this paper can be strengthened or modified. Recall that in Theorem \ref{part(2) with rank}, the abelian groups $A,B,G,H$ that we construct are finitely generated; in fact, they are all infinite cyclic groups. 

\begin{question}Does there exist a similar construction for Theorem \ref{part(2) with generators}? That is, if we require $G$ and $H$ to be finitely generated abelian groups, %but still also require uniformity in the algorithm $F$?
will the resulting $F$ still compute $\emptyset'$, or will it be of strictly lower Turing degree? % does the degree become weaker?
\end{question}

In Theorem 3.1, Theorem 3.2, and Theorem 2.1, we required the rank to be given ahead of time. 

\begin{question} Do these results still hold if the rank is not known? In particular, can the hypothesis of lemma \ref{cyclic decomposition} be weakened to do away with requiring the rank to be given? 
\end{question}
%Does there exist a $\mathbf{0'}$-computable function $F$ that outputs the desired isomorphism in the case of unknown rank?

In the case where the above is possible, one might also ask when we can do away with being given the generators ahead of time as well. 

\begin{question}What, if any, conditions do we require to $\mathbf{0'}$-computably find the generators of a computable finitely generated abelian group?
\end{question}

%For instance: 
% What about if we fix all or some groups but allow their generators to change? One more question is whether lemma 2.5 can be weakened to do away with requiring rank.

%In addition, we reiterate some of the open questions in Deveau's thesis \cite{Deveau} that remain relevant. The troublesome groups still remain to be direct sums of $\Z$, which can be represented as lattices. Are there analogous theorems or insights to be gained when studying lattices, or are there similar lattice problems of interest? Lastly, WCT is related to the projectivity of abelian groups, as the following known theorem states.

% \begin{theorem}[Lubarsky and Richman \cite{Lubarsky}]
% Let $A$ be an abelian group and $f$ and $g$ surjective homomorphisms from $A$ onto $Z$.
% Then $f(\ker(g)) = g(\ker(f))$.
% \end{theorem}

%Deveau has expanded on this question in the thesis, noting the possibility that WCT and this theorem could be equivalent in the sense of Weihrauch reducibility (a way to discuss theorem computational complexity; see Brattka and Gerardi \cite{Brattka}), and in this sense, that only the non-effective portion of the proof is needed.

\bibliographystyle{unsrt}
\bibliography{Reference}

\begin{thebibliography}{1}

\bibitem{Cohn}
P.~M. Cohn.
\newblock The complement of a finitely generated direct summand of an abelian
  group.
\newblock {\em Proc. Amer. Math. Soc.}, 7:520--521, 1956.

\bibitem{Walker}
Elbert~A. Walker.
\newblock Cancellation in direct sums of groups.
\newblock {\em Proc. Amer. Math. Soc.}, 7:898--902, 1956.

\bibitem{Kaplansky}
Irving Kaplansky.
\newblock {\em Infinite abelian groups}.
\newblock University of Michigan Press, Ann Arbor, Mich., revised edition,
  1969.

\bibitem{Knight}
C.~J. Ash and J.~Knight.
\newblock {\em Computable structures and the hyperarithmetical hierarchy},
  volume 144 of {\em Studies in Logic and the Foundations of Mathematics}.
\newblock North-Holland Publishing Co., Amsterdam, 2000.

\bibitem{Deveau}
{Deveau, Michael}.
\newblock {\em Computability Theory and Some Applications}.
\newblock PhD thesis, 2019.

\bibitem{Lang}
Serge Lang.
\newblock {\em Algebra}.
\newblock Addison-Wesley Pub., 1999.

\bibitem{Soare}
Robert~I. Soare.
\newblock {\em Turing computability}.
\newblock Theory and Applications of Computability. Springer-Verlag, Berlin,
  2016.
\newblock Theory and applications.

\end{thebibliography}

\end{document}